\documentclass[preprint,12pt]{elsarticle}

\usepackage{amssymb}
\usepackage{amsmath}

\journal{European Journal of Control}

\usepackage{amsmath, amssymb, bbm, xspace}
\usepackage{amsfonts,mathrsfs}
\usepackage{mathtools}
\usepackage{color}
\usepackage{graphicx}
\usepackage{epsfig}
\usepackage{algorithm}
\usepackage{algorithmicx}
\usepackage[acronym]{glossaries}
\usepackage{subfigure}
\usepackage{cancel}
\usepackage{empheq}
\usepackage{placeins}

\usepackage{enumitem}
\usepackage{textgreek}
\usepackage{mathbbol}
\usepackage{stackengine}
\usepackage{placeins}
\usepackage{dblfloatfix}
\usepackage{cancel}

\def\QEDhereeqn{\eqno\let\eqno\relax\let\leqno\relax\let\veqno\relax\hbox{\QED}}
\def\QEDopenhereeqn{\eqno\let\eqno\relax\let\leqno\relax\let\veqno\relax\hbox{\QEDopen}}
\newcommand{\bs}{\boldsymbol}

\renewcommand{\emph}{\textit}

\newcommand{\0}{\bs 0}
\def\1{{\bs 1}}
\def\argmin{\mathop{\rm argmin}}
\def\min{\mathop{\rm min}}

\newcommand{\proj}{\mathrm{proj}}

\def\R{\mathbb{R}}

\DeclareSymbolFontAlphabet{\mathbbm}{bbold}
\DeclareSymbolFontAlphabet{\mathbb}{AMSb}%

\stackMath
\newcommand\tsup[2][2]{%
	\def\useanchorwidth{T}%
	\ifnum#1>1%
	\stackon[-.5pt]{\tsup[\numexpr#1-1\relax]{#2}}{\scriptscriptstyle\sim}%
	\else%
	\stackon[.5pt]{#2}{\scriptscriptstyle\sim}%
	\fi%
}

\makeglossaries
\newacronym{KKT}{KKT}{Karush--Kuhn--Tucker}
\newacronym{ADMM}{ADMM}{alternating direction method of multipliers}
\newacronym{OPF}{OPF}{optimal power flow}
\newacronym{OPFP}{OPFP}{optimal power flow problem}
\newacronym{NUM}{NUM}{network utility maximization}
\newacronym{LMI}{LMI}{linear matrix inequality}
\newacronym{BMI}{BMI}{bilinear matrix inequality}
\newacronym{LM}{LM}{Lyapunov-Metzler}
\newacronym{SDP}{SDP}{semidefinite programming}
\newacronym{LTI}{LTI}{linear time invariant}
\newacronym{MJLS}{MJLS}{Markov jump linear system}
\newacronym{PID}{PID}{proportional-integral-derivative}
\newacronym{PPA}{PPA}{proximal-point algorithm}
\newacronym{PPPA}{PPPA}{preconditioned proximal-point algorithm}
\newacronym{PPP}{PPP}{preconditioned proximal-point}
\newacronym{NE}{NE}{Nash equilibrium}
\newacronym{GNE}{GNE}{generalized Nash equilibrium}
\newacronym{v-GNE}{v-GNE}{variational GNE}
\newacronym{ISS}{ISS}{input-to-state stability}
\newacronym{OFO}{OFO}{Online Feedback Optimization}

\usepackage{amsthm}

\makeatletter \let\cl@part\relax \makeatother
\usepackage{hyperref}
\usepackage[capitalize]{cleveref}

\usepackage{letltxmacro}

\crefname{thm}{Theorem}{Theorems}
\crefname{lem}{Lemma}{Lemmas}
\crefname{cor}{Corollary}{Corollaries}
\crefname{rem}{Remark}{Remarks}
\crefname{alg}{Algorithm}{Algorithm}
\crefname{figure}{Figure}{Figures}
\crefname{assumption}{Assumption}{Assumptions}
\crefname{standing}{Standing Assumption}{Standing Assumption}
\crefname{corollary}{Corollary}{Corollaries}
\crefname{proposition}{Proposition}{Propositions}
\crefname{problem}{Problem}{Problems}
\crefname{example}{Example}{Examples}
\crefname{definition}{Definition}{Definition}
\crefname{condition}{C}{C}

\crefname{thmlisti}{Theorem}{Theorem}
\crefname{lemlisti}{Lemma}{Lemma}
\crefname{asmlisti}{Assumption}{Assumptions}

\newlist{thmlist}{enumerate}{1}
\setlist[thmlist]{label=(\roman{thmlisti}), ref=\thethm(\roman{thmlisti}),noitemsep}
\newlist{lemlist}{enumerate}{1}
\setlist[lemlist]{label=(\roman{lemlisti}), ref=\thelem(\roman{lemlisti}),noitemsep}
\newlist{asmlist}{enumerate}{1}
\setlist[asmlist]{label=(\roman{asmlisti}), ref=\theassumption(\roman{asmlisti}),noitemsep,nosep,leftmargin=*} 

\newtheorem{lemma}{Lemma}
\newtheorem{theorem}{Theorem}

\newtheorem{assumption}{Assumption}

\newtheorem{corollary}{Corollary}
\newtheorem{proposition}{Proposition}

\usepackage{etoolbox}
\patchcmd{\smallmatrix}{\thickspace}{\kern.5em}{}{}
\newenvironment{smallbmatrix}
{ \left[\begin{smallmatrix}}
	{\end{smallmatrix}\right]}

\usepackage{caption}
\usepackage{booktabs}
\usepackage{multirow}

\begin{document}

\begin{frontmatter}

\title{A Stability Condition for Online Feedback Optimization\\without Timescale Separation}


\author[1]{Mattia Bianchi\corref{cor1}}
\ead{mbianchi@ethz.ch}

\author[1]{Florian D\"orfler}
\ead{dorfler@ethz.ch}

\cortext[cor1]{Corresponding author}

\affiliation[1]{organization={Automatic Control Laboratory, ETH Zürich},country={Switzerland}}

\begin{abstract}
    \gls{OFO} is a  control approach to drive a dynamical plant to an optimal steady state.
    By interconnecting optimization algorithms with real-time plant measurements, 
    \gls{OFO} provides all  the  benefits of feedback control, yet without requiring exact knowledge of plant dynamics for computing a setpoint. On the downside, existing stability guarantees for \gls{OFO} require the controller to evolve on a sufficiently slower timescale than the plant, possibly affecting transient performance and responsiveness to disturbances. In this paper, we prove that, under suitable conditions, \gls{OFO} ensures stability without any timescale separation. In particular, the condition we propose is independent of the time constant of the plant, hence it is scaling-invariant. Our analysis leverages a composite Lyapunov function, which is the $\max$ of plant-related  and controller-related components. 
    We corroborate our theoretical results with   numerical examples. 
\end{abstract}

\begin{keyword}
online feedback optimization  \sep timescale separation   \sep  optimization algorithms  \sep extremum seeking
\end{keyword}

\end{frontmatter}

\section{Introduction}
\glsreset{OFO}\gls{OFO} \cite{Hauswirth2021} is an emerging control paradigm to steer a plant to an efficient equilibrium,  unknown a priori and implicitly defined by an optimization problem. 
In \gls{OFO}, optimization algorithms are used as dynamic feedback controllers, by feeding them with real-time plant measurements. This brings several advantages with respect to feedforward optimization, where plant setpoints are  (periodically) computed offline. In particular, \gls{OFO} showcases superior robustness to model uncertainty and unmeasured disturbances, and it can adapt to unforeseen changes in the plant or cost function. For these reasons, \gls{OFO} has found application in several domains, most notably power
systems (e.g., for frequency regulation  \cite{simpson2020stability,chen2020distributed}
or optimal power flow \cite{Tang2017,Emiliano2016}), but also  process control \cite{ZAGOROWSKA2023119}, traffic
control \cite{bianchin2021time}, communication networks \cite{wang2011control}, even being employed
in industrial setups \cite{Ortmann_Deployment_2023}. 

With regards to the current state-of-the-art, three main limitations emerge for \gls{OFO} methods. 
The first is that  implementing feedback optimization controllers still requires some modeling of the plant (in contrast, for instance, with extremum seeking), in the form of the input-output sensitivity function. This issue is addressed in \cite{he2022,picallo2022adaptive}, by estimating the sensitivity online. The second restriction is  the difficulty of enforcing output constraints. While some approaches ensure  asymptotic satisfaction, by linearizing  \cite{Haberleverena2021} or dualizing \cite{Bernstein2019} the constraints, the work \cite{colot2024optimalpowerflowpursuit} seems to be the first to guarantee transient safety. 
The third limitation is that stability guarantees for \gls{OFO} schemes assume timescale separation, i.e., that the physical plant evolves on a faster timescale than the optimization algorithm/controller. The latter issue is the focus of this paper. 

In particular, most works in the field assume that the plant is stationary, i.e., they identify  the plant with its steady-state map \cite{Hauswirth2021,colot2024optimalpowerflowpursuit}. This is a good approximation for extremely  fast plants (e.g., for frequency regulation in power grids). Otherwise, one has to take into account that the dynamics of the controller can interfere with the plant dynamics, as it was done in \cite{menta2018stability,Hauswirthadrian2021,Belgioioso_FES_2024}, \cite[§3.B]{Colombino2020}. In these papers, stability of the closed-loop system is proven by enforcing that the controller is sufficiently slower (i.e., has a much larger time constant) than the plant, as in classical singular perturbation arguments. In practice, this means updating the dynamic controller (namely, the optimization routine) at a very slow rate, or with a small stepsize. Indeed,  in some cases a fast controller could lead to instability \cite{Hauswirthadrian2021,Belgioioso_FES_2024}.

On the other hand, in most setups it would be desirable to operate the controller on the same timescale of the plant \cite{Picallo2023}, to improve transient performance and reduce settling times, especially in problems with high temporal variability. The questions to answer are thus how and under which conditions can stability be guaranteed without timescale separation,  by designing  new  \gls{OFO} schemes or via novel analysis methods, respectively.  Here we take the second route. 

\emph{Contributions}: In this paper, we show that
timescale separation is not always necessary to guarantee stability with \gls{OFO} controllers. Specifically, we derive a stability condition that does not require the controller to be slower (nor faster) than the plant, and that ensures exponential stability for the closed-loop system, for any choice of the control gain. We further provide sufficient conditions for our stability criterion, and we show that they can be  enforced by adding sufficient regularization to the objective function of the optimization problem (at the cost of suboptimality). 

Our stability condition is conceptually related to the (block) diagonal dominance  of the Jacobian of the closed-loop, and our analysis is based on a $\max$ Lyapunov function, reminiscent of arguments used in
asynchronous iterations \cite{Bertsekas_Parallel}, switched systems stability \cite{Hu_TAC2008}, distributed algorithms \cite{Nguyen2023}, contraction theory \cite{Davydov_NonEuclideanContraction_2022}---but with different goals. Here we use a  $\max$-type Lyapunov function to avoid timescale separation, and we further provide a general result  which allows for arbitrary Lyapunov functions for the subsystems (see \Cref{lem:fundamental}). 
For ease of reading, we tailor our analysis to an \gls{OFO} controller based on the simplest continuous-time gradient flow, although our approach can be  generalized in several directions, and it is not limited to \gls{OFO}. 

The paper is organized as follows. In \Cref{sec:problemstatement}, we review \gls{OFO} controllers. \Cref{sec:mainresults} introduces our main condition and stability theorem. In \Cref{sec:constrained}, we extend the result to the input-constrained case. \Cref{sec:numerics}  illustrates our results via numerical examples. Finally, in \Cref{sec:conclusion}, we discuss possible extensions and directions for future research.





\subsection{Notation and preliminaries}
For a differentiable map $f:\R^n \rightarrow \R^m$, 
$\nabla f(x) \in \R^{m\times n}$ denotes its Jacobian, i.e.,  the matrix of partial derivatives of $f$ computed at $x\in \R^n$. If $f$ is scalar (i.e., $m=1$), we also denote by $\nabla f:\R^n \rightarrow \R^n $ its gradient, with some abuse of notation but no ambiguity. If $x'$ is a subset of variables, then $\nabla_{x'} f$ denotes its Jacobian/gradient with respect to $x'$.
We say that a function $V:\R_{\geq 0}\rightarrow \R$ is differentiable if it is differentiable on the open set $\R_{>0}$; $\dot V(t)$ denotes its derivative computed at $t$. Given a dynamical system $\dot x(t) = f(x)$, if $V:\R^n \rightarrow\R$, then $\dot V(x(t)) \coloneqq \nabla V (x(t))^\top f(x(t))$ denotes its Lie derivative along the vector field $f$.  The (right) Dini derivative of a function $V: \R_{\geq 0} \rightarrow \R$ is defined as
\begin{align}
	D^+ V(t) = \underset{\delta \rightarrow 0^+ }{\operatorname{lim \ sup}} \ \frac{V(t+ \delta)-V(t)}{\delta}.
\end{align}
\begin{lemma}[Danskin’s lemma \cite {Davydov_NonEuclideanContraction_2022}]\label{lem:Danskin}
	For some differentiable functions $V_i:\R_{\geq 0} \rightarrow \R$, $i =1,\dots, N$, let $V = \max\limits_{i \in \{ 1,\dots,N\}} \{V_i\}$ . Then, for all $t >0$, 
	\begin{align}
		D^+ V(t) =  \max_{i \in \{ 1,\dots,N\}} \{ \dot V_i(t) \mid V(t) = V_i(t) \}.
	\end{align}
\end{lemma}
\medskip  


\begin{lemma}[{\cite[Lem.~11]{Davydov_NonEuclideanContraction_2022}}]\label{lem:Gronwall}
	Let $V \! : \! \R_{\geq 0} \! \rightarrow \! \! \R$ be continuous. If $D^+ V(t) \! \leq  \!\!  - \tau V(t)$ for almost all $t$, then, for all $t$, 
	\begin{align}
		V(t) \leq e^{-\tau t} V(0). 
	\end{align}
\end{lemma}

\medskip

\section{\gls{OFO} and timescale separation} \label{sec:problemstatement}
In this section, we review the main idea of \gls{OFO} and discuss the role of timescale separation for this control approach. 
We consider the dynamical system (or plant)
\begin{subequations}\label{eq:plant}
	\begin{align}
		\dot x & = f(x,u) \\
		y & = g(x),
	\end{align}
\end{subequations}
with state $x\in\R^n$, input $u\in \R^m$, and output $y\in \R^p$. In the following, we postulate  that the plant is ``pre-stabilized'' and that it has a well-behaved steady-state map; this is a fundamental assumption for \gls{OFO}  \cite[Asm.~2.1]{Hauswirthadrian2021}. 

\begin{assumption}\label{asm:plant}
	The mappings $f:\R^n \times \R^m \rightarrow \R^n$ and $g:\R^n \times \R^m \rightarrow \R^p$ are locally Lipschitz. For any constant input $u(t) = u\in \R^m$, the plant \eqref{eq:plant} is globally asymptotically stable; hence,  there are a unique steady-state map $s:\R^m \rightarrow \R^n$ and a steady-state output map $h:\R^m \rightarrow \R^p$ such that
	\begin{align}\label{eq:h}
		f(s(u),u) = 0, \quad  h(u) \coloneqq g(s(u)),
	\end{align}
	for all $u\in\R^m$. Furthermore, $h$ is continuously differentiable and the \emph{sensitivity} $\nabla h:\R^m \rightarrow \R^{p \times m}$ is locally Lipschitz.   
\end{assumption}

The control objective is to steer the plant \eqref{eq:plant} to a solution of the optimization problem 
\begin{subequations}
	\label{eq:opt1}
	\begin{align}
		\min_{u\in\R^m,y\in\R^p} & \quad   \Phi(u,y)
		\\[-0.3em]
		\textnormal{s.t.}   & \quad   y = h(u),
	\end{align}
\end{subequations}
where $\Phi$ is a cost function, and the steady-state constraint  $y = h(u)$ ensures that any solution to \eqref{eq:opt1} is an input-output equilibrium for the plant \eqref{eq:plant}. Equivalently, \eqref{eq:opt1} can be recast as the unconstrained problem
\begin{align}\label{eq:opt2}
	\min_{u\in\R^m} & \quad  \tilde \Phi(u)
\end{align} 
where $\tilde \Phi(u)  \coloneqq \Phi(u,h(u))$.

\begin{assumption}\label{asm:cost}
	The function $\Phi:\R^m\times \R^p \rightarrow \R$ is continuously differentiable, and its gradient  $\nabla \Phi$ is locally Lipschitz. The problem in \eqref{eq:opt1} admits a solution.
\end{assumption}

Existence of solutions to \eqref{eq:opt1} is ensured if $\tilde \Phi$ has compact level sets (e.g., if it is strongly convex) as it was  assumed, e.g, in  \cite{menta2018stability,Hauswirthadrian2021}. Note that the chain rule implies
\begin{align}
	\nabla \tilde \Phi (u)  = \tilde H(u) \nabla \Phi (u,h(u)).
\end{align}
where $\tilde H(u) \coloneqq \begin{bmatrix}
	I_m &  \nabla h(u)^\top
\end{bmatrix}$. Hence, the gradient flow for the unconstrained problem \eqref{eq:opt2} would be 
\begin{align}\label{eq:openloopGradient}
	\dot u = - \alpha \tilde H(u) \nabla \Phi (u,h(u)),
\end{align}
with tuning gain $\alpha >0$. 
\gls{OFO} proposes to replace this flow with a feedback controller, obtained by substituting the steady-state map $h(u)$ with the current measured output of the plant \eqref{eq:plant}, namely 
\begin{align}\label{eq:OFO}
	\dot u = - \alpha \tilde H (u) \nabla \Phi (u,y).
\end{align}
Besides ensuring the usual advantages of feedback control  (e.g.,  in terms of disturbance rejection), implementing \eqref{eq:OFO} does not require knowing $h$, but only the so-called \emph{sensitivity} $\nabla h$, which is easier to estimate online \cite{picallo2022adaptive}, and  it is independent of unknown additive terms affecting $h$ (e.g., a constant disturbance, as for example the unknown demand in power systems applications \cite{colot2024optimalpowerflowpursuit}; see also \Cref{sec:numerics}).

The equilibria of the closed-loop system \eqref{eq:plant}, \eqref{eq:OFO} correspond to the critical points of \eqref{eq:opt2}. 
Informally speaking, if the plant is close to steady state (i.e., $y \approx h(u)$), then  we have $\dot u \approx - \alpha \nabla \tilde \Phi (u)$, which is the standard gradient flow. In fact, stability of the closed-loop can be guaranteed\footnote{Under suitable conditions: for instance, it is usually assumed that the plant \eqref{eq:plant} is exponentially stable for any constant input.}  by making the time constant of the controller much larger than that of the plant, so that the plant is always approximately at steady state  \cite{Hauswirthadrian2021,Colombino2020}, as typical in singular perturbation analysis. 

In practice, this timescale separation is achieved by choosing a small-enough gain $\alpha$ for the controller. Nonetheless,  a small $\alpha$ affects the response speed of the controller \cite{Picallo2023}, for instance resulting in worse transient performance, longer settling time, poor responsiveness to changes in the operation (e.g., for time-varying costs or disturbances). Hence, it would be highly desirable to guarantee closed-loop stability when controller and plant evolve on the same timescale.


\section{Stability without timescale separation} \label{sec:mainresults}
In this section, we show that under suitable conditions the \gls{OFO} controller  \eqref{eq:OFO} is stabilizing for the plant \eqref{eq:plant} without any timescale separation, namely \emph{for any value of} $\alpha$. 

\subsection{Main idea}
We start by providing some intuition on our derivation. Consider the closed-loop system
\begin{align} \label{eq:closedloop}
	\dot \omega = 
	\begin{bmatrix}
		\dot x
		\\
		\dot u
	\end{bmatrix}    = \begin{bmatrix}
		f(x,u)
		\\
		- \alpha \tilde H (u) \nabla \Phi (u,(g(x))
	\end{bmatrix}\coloneqq \mathcal A(x,u).
\end{align}
Let us denote $k(x,u) \coloneqq -\tilde H (u) \nabla \Phi (u,(g(x)) $; under sufficient differentiability, the Jacobian of $\mathcal A$ is
\begin{align*}
	\nabla \mathcal A(x,u)  = \begin{bmatrix}
		\nabla_x f(x,u) & \nabla_u f(x,u)
		\\
		\alpha \nabla_x k(x,u)  &  \alpha \nabla_u k(x,u)
	\end{bmatrix}.
\end{align*}
One way to verify stability is to show that there exists $Q \succ 0$ satisfying the Lyapunov inequality
\begin{align*}
	(\forall x\in\R^n,\forall u\in\R^m) \quad  \nabla \mathcal A(x,u) Q + Q\nabla \mathcal A(x,u) ^\top\prec 0.
\end{align*}
In this case, $V(\omega) = (\omega -\omega^\star)^\top  Q (\omega-\omega^\star)$  provides a Lyapunov function for the closed-loop system, where $\omega^\star = (x^\star ,u^\star)$,  $u^\star$ is a solution to \eqref{eq:opt2}, and $x^\star = s(u^\star)$. However, this stability condition inevitably depends on the value of $\alpha$. An alternative is to look at the diagonal dominance of $\nabla \mathcal A$: if the diagonal terms are negative  and the off-diagonal terms are sufficiently small, it is possible to show  stability via an infinity-norm Lyapunov function  \cite{Davydov_NonEuclideanContraction_2022}. The advantage is that diagonal dominance  does not depend on the value of  $\alpha$, and thus on timescale separation. 

While this condition on the Jacobian would result in very restrictive stability conditions, the diagonal dominance argument inspires the following fundamental result, which is the cornerstone of our analysis.  
\begin{lemma}\label{lem:fundamental}
	Let $\xi>0$, $\alpha>0$ be arbitrary. Assume that   $V_x :\R_{\geq 0}\rightarrow \R_{\geq 0} $ and $V_u :\R_{\geq 0}\rightarrow \R_{\geq 0} $ are  differentiable  nonnegative functions satisfying, for almost all $t\geq 0$,
	\begin{align}
		\dot V_x (t) &  \leq - \mu_1 V_x(t)+ \theta_1 V_u(t)
		\\
		\dot V_u(t) &  \leq \alpha \left( \theta_2 V_x(t) - \mu_2 V_u (t)\right),
	\end{align}
	for some $\mu_1,\mu_2,\theta_1,\theta_2 >0$ such that \begin{align*}
	    -\mu_1 + \xi \theta_1 <0, \qquad \theta_2 - \xi \mu_2 <0.
        \end{align*} Then, $V \coloneqq \max\{\xi V_x,V_u\}$ satisfies 
        \begin{align*}
            V(t) \leq e^{-\tau t} V(0),
        \end{align*}  for all $t\geq 0$, where $\tau = \min\{ \mu_1-\xi\theta_1, \alpha(  \mu_2 - \frac{\theta_2}{\xi})\} >0$. 
\end{lemma}

\begin{proof}
   Since $V$ is possibly not differentiable at time $t$  if $\xi V_x(t) = V_u(t)$, we use the Dini derivative. By \Cref{lem:Danskin}, if $\xi V_x(t) > V_u(t)$, then $ D^+ V(t) = \xi \dot V_x(t)  \leq  -\mu_1 \xi V_x(t) + \xi \theta_1 V_u(t) \leq - \mu_1 V(t) + \xi \theta_1 V(t) \leq -\tau V(t)$. Similarly, if $\xi V_x(t) < V_u(t)$,  we have $D^+ V(t) = \dot V_u(t) \leq  \alpha (\theta_2 V_x(t) - \mu_2 V_u(t)) \leq \alpha ( \frac{\theta_2}{\xi}  V(t) - \mu_2 V(t)) \leq - \tau V(t)$. Finally, still by \Cref{lem:Danskin}, if $\xi V_x(t) = V_u(t)$, then $D^+ V(t) = \max \{\xi \dot V_x, \dot V_u \} \leq   \max \{ -\mu_1\xi V_x(t) + \xi \theta_1 V_u(t), \alpha(\theta_2 V_x(t) - \mu_2 V_u(t)) \} \leq \max \{ - \tau V(t), -\tau V(t) \} = -\tau V(t)$. We conclude that, for all $t>0$,
   $D^+ V(t) \leq -\tau V(t)$.
   The conclusion follows by \Cref{lem:Gronwall}.
\end{proof}

Note that $\alpha$ in \Cref{lem:fundamental}  is arbitrary. The relations $-\xi \mu_1 +  \theta_1 <0$ and   $-\mu_2+ \xi \theta_2 <0$ can be seen as block ``diagonally dominance'' conditions, weighted by the parameter $\xi$, which provides an extra degree of freedom.  
Our goal is of course to apply \Cref{lem:fundamental} to the Lyapunov analysis of \eqref{eq:closedloop}. 

\subsection{Main result}

We start by formulating our main convergence condition. 

\begin{assumption}\label{asm:core}
	Let $u^\star$ be a critical point of \eqref{eq:opt2}, and let $x^\star = s(u^\star)$.  With reference to the dynamics in \eqref{eq:closedloop}, there exist positive constants $c_1,d_1,\mu_1,\theta_1$, $c_2,d_2,\mu_2,\theta_2,\xi >0$ and continuously differentiable functions $V_x:\R^n \rightarrow \R_{\geq 0}$ and $V_u:\R^m \rightarrow \R_{\geq 0}$,  such that, for all $x \in \R^n$ and $ u\in \R^m$,
	\begin{align*}
		c_1\|x-x^\star\|^2  & \leq V_x(x) \leq d_1 \| x-x^\star\|^2 
		\\
		c_2\|u-u^\star\|^2 &  \leq V_u(u) \leq d_2 \| u-u^\star\|^2,
	\end{align*}
	and moreover:
	\begin{itemize}
		\item[(i)] \emph{Plant robust stability}: $\forall x,u$, it holds that
		\begin{align}
			\dot V_x(x) \leq -\mu_1 V_x(x) + \theta_1 V_u(u);
		\end{align}
		\item[(ii)] \emph{Algorithm robust stability}:  $\forall x,u$, it holds that 
		\begin{align}
			\dot V_u(u) \leq \alpha \left (-\mu_2 V_u(u) + \theta_2 V_x(x) \right) ;
		\end{align}
		\item[(iii)] \emph{Parameter dominance}: 
		It holds that $-\mu_1 + \xi \theta_1 <0$ and   $  \theta_2 - \xi \mu_2 <0$. 
	\end{itemize}
\end{assumption}

We postpone a detailed discussion of \Cref{asm:core} to \Cref{sec:OnAssumption}, after presenting our main result. 
\medskip 

\begin{theorem}\label{th:core}
	Let Assumptions~\ref{asm:plant}, \ref{asm:cost}, and \ref{asm:core} hold.  Then, for any $\alpha >0$, the closed-loop system \eqref{eq:closedloop} is globally exponentially stable, and $(u,y)$ converges to the unique solution $(u^\star,y^\star)$ of \eqref{eq:opt1}.
\end{theorem}

\begin{proof}
	By the local Lipschitz conditions in \Cref{asm:plant,asm:cost}, the system in \eqref{eq:closedloop} admits a unique (local) solution, for any initial condition. Consider the candidate Lyapunov function $V(x,u)= \max \{ \xi V_x(x), V_u(u)\}$, $V_x$ and $V_u$ as in \Cref{asm:core}. Then, by \Cref{lem:fundamental}, $V(x(t),u(t)) \leq e^{-\tau t} V(x(0),u(0))$. Since $V$ is continuous with bounded level sets and is monotonically decreasing to zero along the solutions of \eqref{eq:closedloop}, we can immediately conclude that  \eqref{eq:closedloop} has a unique complete solution and is asymptotically stable. Since $V(x,u)$ is radially unbounded, the result holds globally. By \Cref{asm:core} and definition of $V$, $ \frac{\xi c_1}{2}\|x-x^\star\|^2 +\! \frac{c_2}{2}\|u-u^\star\|^2 \! \leq \! V(x,u) \! \leq  \! d_1 \xi \|x-x^\star\|^2 
	+ d_2\|u-u^\star\|^2$, hence convergence is exponential. The fact that  $x(t) \rightarrow x^\star$, $u(t) \rightarrow u^\star$ as $t\rightarrow \infty$  implies  $y(t) = g(x(t)) \rightarrow g(x^\star) = g(s(u^\star)) = h(x^\star) = y^\star$  by continuity of $g$. Since the result is global, $(x^\star,u^\star)$ must be the unique equilibrium of \eqref{eq:closedloop}, thus \eqref{eq:opt1} has a unique critical point, thus a unique solution (one solution exists by \Cref{asm:cost}). 
\end{proof}

\subsection{On \Cref{asm:core}}\label{sec:OnAssumption}

\Cref{asm:core} is novel and deserves some discussion. In the following, we provide  insight by relating it to some standard conditions used in the \gls{OFO} literature, collected in \Cref{asm:simplifying} below. We will show that the common conditions in \Cref{asm:simplifying}  are more restrictive than the novel conditions we introduced in  \Cref{asm:core}(i)-(ii).  

\begin{assumption}\label{asm:simplifying}
	The following holds:
	\begin{itemize}
		\item[(a)] The map $f(x,u)$ in \eqref{eq:plant} is $\ell_f$-Lipschitz in $u$ for any fixed $x\in\R^n$; the map $g$ in \eqref{eq:plant}  is $\ell_g$-Lipschitz continuous; 
		\item[(b)] With reference to the plant in \eqref{eq:plant}, there exist constants $c_3,d_3,\mu_3,\zeta_3 >0$ and  a function $W:\R^n\times \R^m \rightarrow \R$, such that, for any constant input $u(t) =u\in\R^m$, and any $x\in\R^n$,
		\begin{align*}
			c_3 \| x- s(u)\|^2 \leq W(x,u) & \leq d_3 \| x- s(u)\| ^2
			\\ \nabla_x W(x,u)^\top f(x,u) & \leq - \mu_3  \|x-s(u)\|^2
			\\
			\| \nabla_x W (x,u)\| & \leq \zeta_3\| x-s(u) \|;
		\end{align*}
		\item[(c)] The cost function $\Phi(u,y)$ in \eqref{eq:opt1} is $\mu_\Phi$-strongly convex in $u$ for any fixed $y\in \R^p$; the map $\tilde H(u)\nabla \Phi(u,y) $  is $\ell_{\Phi_y}$-Lipschitz continuous in $y$ for any fixed $u\in \R^m$;
		\item[(d)] The map $\nabla h(u)^\top \nabla_y \Phi(u,y)$ is $\ell_{\Phi_u}$-Lipschitz continuous in $u$, for any fixed $y\in \R^p$, and it holds that $\mu_\Phi > \ell_{\Phi_u}$.  
	\end{itemize}
\end{assumption}

\medskip
Let us first compare \Cref{asm:core}(i)
and \Cref{asm:simplifying}(b). \Cref{asm:simplifying}(b) is rather standard in the \gls{OFO} literature; it stipulates that the plant \eqref{eq:plant} converges exponentially to its steady state $s(u)$ for any constant input (see \cite[Th.~5.17]{Sastry} or \cite[Prop.~2.1]{Hauswirthadrian2021}). Instead, \Cref{asm:core}(i) is an input-to-state stability property with respect to the input $(u-u^\star)$; it implies exponential convergence only when $u = u^\star$, rather than for any arbitrary input (although \emph{asymptotic} stability for fixed input is assumed in \Cref{asm:plant}). In fact, in \Cref{prop:parameters} below we establish that,  under \Cref{asm:simplifying}(a) (which is also  standard \cite[Asm.~2.1]{Hauswirthadrian2021}), \Cref{asm:simplifying}(b) implies \Cref{asm:core}(i). 

We will also show that  \Cref{asm:simplifying}(c)-(d) imply \Cref{asm:core}(ii). In \Cref{asm:simplifying}(c), the smoothness is a common condition. Strong convexity was assumed for instance in \cite{Colombino2020,Belgioiosogiuseppe2021}. We should emphasize that several results on \gls{OFO} stability are available also for  general smooth nonconvex costs, under timescale separation \cite{Hauswirthadrian2021}: however, in this case it  appears difficult to ensure a property like \Cref{asm:core}(ii) (if not resorting to extra conditions, e.g., error bounds). \Cref{asm:simplifying}(d) is for instance automatically satisfied with $\ell_{\Phi_u} =0$ whenever $\Phi(u,y) = \Phi_u(u)+\Phi_y(y)$ is separable and  $h$ is a linear map (e.g., for linear plants, which are the most commonly studied in the literature \cite{Colombino2020,menta2018stability,Bernstein2019}).

\begin{proposition}\label{prop:parameters}
	Let \Cref{asm:simplifying} hold. Then \Cref{asm:core}(i) and \ref{asm:core}(ii) hold with $V_x(x) = W(x,u^\star)$, $V_u = \frac{1}{2}\| u- u^\star\|^2$, $\mu_1 = \frac{\mu_3}{2 d_3}$, $\theta_1 = \frac{\ell_f^2 \zeta_3^2}{2 \mu_3}$, $\theta_2 =\frac{\ell_g^2 \ell_{\Phi_y}^2}{2(\mu_\Phi-\ell_{\Phi_u}) c_3}$, and $\mu_2 =  \frac{(\mu_\Phi-\ell_{\Phi_u})}{2}>0.$
\end{proposition}
\medskip 
\begin{proof}
	Let $V_x(x) = W(x,u^\star)$ and $V_u(u) = \frac{1}{2}\|u-u^\star\|^2$  (where $u^\star$ can be any critical point of  \eqref{eq:opt2}). For the first part, along \eqref{eq:closedloop}, we have
	\begin{align*}
		\dot V_x(x) & = \nabla_x W(x,u^\star)^\top f(x,u)
		\\
		& = \nabla_x W(x,u^\star)^\top (f(x,u^\star) + f(x,u)-f(x,u^\star))
		\\
		& \textstyle
		\leq  -\mu_3 \|x- s(u^\star) \|^2 + \ell_f \zeta_3 \|x-s(u^\star)\|\|u- u^\star\|
		\\ 
		&  \textstyle \overset{(i)} \leq -\frac{\mu_3}{2}\|x- s(u^\star) \|^2 + \frac{\ell_f^2 \zeta_3^2}{2 \mu_3}\|u- u^\star\|^2
		\\ 
		& \textstyle \leq - \frac{\mu_3}{2 d_3} W(x,u^\star) + \frac{\ell_f^2 \zeta_3^2}{2 \mu_3} V_u(u),
	\end{align*}
	where we used Young's inequality in (i). For the second part, we have 
	\begin{align*} \dot V_u(u)  & = -\alpha(u-u^\star)^\top ( \tilde H(u)\nabla \Phi (u,y)) \\ & = -\alpha(u-u^\star)^\top ( \nabla_u \Phi (u,y) + \nabla h(u)^\top \nabla_y \Phi(u,y)
    \\ & \hphantom{{}={}}
	\pm
	\nabla_u \Phi(u^\star,y) \pm \nabla h(u^\star)^\top \nabla_y \Phi(u^\star,y) 
	\\ & \hphantom{{}={}} - \nabla_u \Phi(u^\star,y^\star)- \nabla h(u^\star)^\top \nabla_y \Phi(u^\star,y^\star)),\end{align*} where $y^\star \coloneqq h(u^\star)$. In the last equality we added an subtracted some terms in order to be able to use the bounds in  \Cref{asm:simplifying}, and furthermore we used that $u^\star$ is a critical point of \eqref{eq:opt2}, i.e.,   $\nabla_u \Phi(u^\star,y^\star)+ \nabla h(u^\star)^\top \nabla_y \Phi(u^\star,y^\star)= 0$. Hence we have
	\begin{align*}
		\dot V_u(u) & \hspace{-0.2em}\leq \hspace{-0.2em} -\alpha (\mu_\Phi-\ell_{\Phi_u})  \|u - u^\star\|^2 \hspace{-0.2em}+ \hspace{-0.2em}\alpha \ell_{\Phi_y}  \|u-u^\star\| \|y-y^\star\|
		\\ 
		& \textstyle \overset{(ii)}\leq - \alpha \frac{\mu_\Phi-\ell_{\Phi_u}}{2}\|u-u^\star\|^2 +\alpha \frac{\ell_{\Phi_y}^2 \ell_g^2}{2 (\mu_\Phi-\ell_{\Phi_u})} \| x-x^\star\|^2
		\\
		& \textstyle  \leq -\alpha \frac{(\mu_\Phi-\ell_{\Phi_u})}{2} V_u(u)+ \frac{\ell_{\Phi_y}^2 \ell_g^2}{2(\mu_\Phi-\ell_{\Phi_u}) c_3}  V_x(x),
	\end{align*}
	where we used Young's inequality in (ii).
\end{proof}

\begin{corollary} \label{cor:first}
	Let \Cref{asm:plant}, \ref{asm:cost} and \ref{asm:simplifying} hold, and let $\mu_1,\theta_1,\theta_2,\mu_2$ be as in \Cref{prop:parameters}. If there is $\xi>0$ such that $-\mu_1 + \xi \theta_1 <0$ and   $  \theta_2 - \xi \mu_2 <0$, then,  for any $\alpha>0$, the closed-loop system \eqref{eq:closedloop} is globally exponentially stable, and $(u,y)$ converges to the unique solution $(u^\star,y^\star)$ of \eqref{eq:opt1}.
\end{corollary}

\begin{proof}
    A direct consequence of \Cref{th:core} and \Cref{prop:parameters}.
\end{proof}

\begin{corollary} \label{rem:enforcing}
	Let $\mu_1$, $\theta_1$, $\theta_2$, and
	$\mu_2$  be as in \Cref{prop:parameters}. Under \Cref{asm:simplifying},
	the dominance condition in \Cref{asm:core}(iii), i.e., $ -\mu_1 + \xi \theta_1 <0$ and $\theta_2 - \xi \mu_2 <0$,     holds for some $\xi>0$ if
	\begin{align}\label{eq:muBound}
		\mu_\phi > \ell_{\Phi_u}+ \sqrt{ \frac{\ell_g^2 \ell_{\Phi_y}^2 d_3 \zeta_3^2 \ell_f^2 }{c_3\mu_3^2}}.
	\end{align}
\end{corollary}
\vspace{1em}

\begin{proof}
	By the expression of   $\mu_1,\theta_1,\theta_2,
	\mu_2$  in \Cref{prop:parameters}, \Cref{asm:core}(iii) is satisfied when $-\frac{\mu_3}{2d_3} + \xi \frac{\ell_f^2 \zeta_3^2}{2\mu_3} < 0$ and $\frac {\ell_g^2 \ell_{\Phi_{y}}^2} { 2 ( \mu_\Phi - \ell_{\Phi_{u}}) c_3}  - \xi \frac{\mu_\Phi - \ell_{\Phi_{u}}}{2} <0 $. The first inequality imposes $\xi < \frac{ \mu_3^2 }{\ell_f^2 \zeta_3^2 d_3}$, while the second reads $( \mu_\Phi - \ell_{\Phi_{u}})^2 > \frac{\ell_g^2 \ell_{\Phi_y}^2}{c_3 \xi}$. Combining the two gives the bound in \eqref{eq:muBound}.
\end{proof}

Notably, the condition in \eqref{eq:muBound} can be always enforced, by replacing the original cost  $\Phi$ with a regularized cost $\bar \Phi(u,y) = \Phi(u,y) + \rho(u)$, where $\rho$ is a sufficiently strongly convex function (e.g., $\rho(u) = \frac{\mu_4}{2} \|u\|^2$, with $\mu_4>0$ large enough so that $\mu_\Phi+\mu_4$ is greater than the right-hand side of \eqref{eq:muBound}). This highlights a trade-off between steady-state and transient performance of the controller: at the price of some suboptimality, arbitrary values of $\alpha$ are allowed ---instead of requiring a small-enough $\alpha$, as in the previous literature.

\section{Input constraints} \label{sec:constrained}
In this section, we assume that the input $u$ is constrained to be in a closed, convex set $\mathcal{U} \subset \R^m$; the objective is hence to drive the plant to the steady-state solutions of 
\begin{align}\label{eq:optConst}
	\min_{u\in\mathcal U,y\in\R^p} & \quad   \Phi(u,y) \qquad 
	\textnormal{s.t.}  \   y = h(u).
\end{align}
We can address input constraints by replacing \eqref{eq:OFO} with the smooth dynamics
\begin{align}\label{eq:OFOConstrainedSmooth}
	\dot u = -\alpha u+\alpha\operatorname{proj}_{\mathcal{U}}  \left( u - \beta \tilde H(u)\nabla \Phi(u,y)\right),
\end{align}
where  $\proj_{\mathcal{U}} (v)  =  \argmin_{u\in\mathcal{U}}  \|u-v\|  $ is the Euclidean projection and  $\beta>0$ is a stepsize. With the same arguments of \Cref{th:core}, we can still conclude the stability of the closed-loop \eqref{eq:plant}, \eqref{eq:OFOConstrainedSmooth}  under a condition  analogous to \Cref{asm:core}. 
Interestingly, \Cref{asm:simplifying}  still suffices to ensure stability, under the same bound \eqref{eq:muBound}.

\begin{corollary}\label{cor:projected}
	Let Assumptions \ref{asm:plant}, \ref{asm:cost}, and \ref{asm:simplifying} hold, and let $\mu_1,\theta_1,\theta_2,\mu_2$ be as in \Cref{prop:parameters}. Assume that $\mathcal{U}$ is closed convex, that \eqref{eq:optConst} admits a solution, that $\nabla_u \Phi(u,y)$ is $L$-Lipschitz in $u$ for any fixed $y \in \R^p$, and that $\beta \leq \frac{1}{L}$. Assume that there is $\xi>0$ such that $-\mu_1 + \xi \theta_1 <0$ and   $  \theta_2 - \xi \mu_2 <0$. Then,  for any $\alpha>0$, the closed-loop system \eqref{eq:plant}, \eqref{eq:OFOConstrainedSmooth} is globally exponentially stable, and $(u,y)$ converges to the unique solution $(u^\star,y^\star)$ of \eqref{eq:optConst}. Furthermore, if $u(0)$ in $\mathcal U$, then $u(t) \in \mathcal U$ for all $t \geq 0$. 
\end{corollary}

\begin{proof}
	We recall the nonexpansivity property of the Euclidean projection, namely that $\| \proj_{\mathcal{U}}(u) - \proj_{\mathcal{U}}(u') \| \leq \| u-u' \|$ for all $u,u'$.  Let $V_u(u) = \frac{1}{2}\| u -u^\star\|^2$, where $u^\star$ is any critical point of \eqref{eq:optConst}, and $y^\star = h(u^\star)$. We have
	\begin{align*}
		\dot V_u(u) & \overset{(i)}= -\alpha (u-u^\star)^\top [ u+\operatorname{proj}_{\mathcal{U}} ( u - \beta \tilde H(u)\nabla \Phi(u,y))
		\\ 
		& \hphantom{{}={}}
		\pm \operatorname{proj}_{\mathcal{U}} ( u^\star- \beta \tilde H(u^\star)\nabla \Phi(u^\star,y))
		\\
		&   \hphantom{{}={}}
		- u^\star + \operatorname{proj}_{\mathcal{U}} ( u^\star- \beta \tilde H(u^\star)\nabla \Phi(u^\star,y^\star))] 
		\\[-0.5em]
		&
		\overset{(ii)}\leq -\alpha\|u-u^\star\|^2 + \alpha \|u-u^\star\| (\cdot),
	\end{align*}
	where $-u^\star + \operatorname{proj}_{\mathcal{U}} ( u^\star- \beta \tilde H(u^\star)\nabla \Phi(u^\star,y^\star)) = 0$ is used  in (i), 
	and we used the nonexpansivity of the projection, the Cauchy Schwartz inequality, and the definition of $\tilde H$ in (ii). The term $(\cdot)$ 
	is
	\begin{align*}
		&  (\cdot) = 
		\|u- \beta \nabla_u \Phi(u,y) - u^\star+ \beta\nabla_u \Phi(u^\star,y)\|
		\\
		&  \hphantom{{}={}}  +\beta \|\nabla h(u)^\top \nabla_y \Phi(u,y) - \nabla h(u^\star)^\top \nabla_y \Phi(u^\star,y)\| 
		\\
		& \hphantom{{}={}}   + 
		\beta \| \tilde H(u^\star)\nabla\Phi(u^\star,y)- \tilde H(u^\star) \nabla\Phi(u^\star,y^\star ) \| 
		\\ & 
		\overset{(iii)}\leq (1-\beta \mu_\Phi) \|u-u^\star\| + \beta \ell_{\Phi_u} \|u-u^\star\| + \beta \ell_{\Phi_y} \|y-y^\star\|,
	\end{align*}
	where we used the  the contractivity of the gradient method for the first term in (iii). Summing up, we obtain  the same bound for $\dot V_u(u)$ as in the proof of \Cref{prop:parameters}, with the only difference of having the factor $\alpha\beta$ instead of $\alpha $ (which is irrelevant, since $\alpha$ is arbitrary). Then the exponential stability results follows as for \Cref{prop:parameters} and \Cref{th:core}. For the second statement, let $\mathrm T_{{\mathcal{U}}}(u) \coloneqq  \operatorname{cl}({\bigcup_{\delta>0}\frac{1}{\delta}(\mathcal U-u)})$ be the tangent cone of $\mathcal U$ at $u$, where $\operatorname{cl}(\cdot)$ is the set closure.
    By definition of $T_{{\mathcal{U}}}$, for the flow in \eqref{eq:OFOConstrainedSmooth}, if $u\in\mathcal U$, then $\dot u \in \mathrm T_{{\mathcal{U}}}(u)$. Hence the set $\mathcal U$ is invariant for \eqref{eq:OFOConstrainedSmooth}.
\end{proof}

It can be similarly shown that, under the very same conditions,  the nonsmooth projected controller $
\dot u =   \proj_{\mathrm T_{{\mathcal{U}}}(u)} (-\alpha \tilde H(u)\nabla \Phi(u,y)) $
(where $\mathrm{T}_{\mathcal U}(u)$ is the tangent cone of $\mathcal{U}$ at $u$) is also exponentially stabilizing. As usual in \gls{OFO}, handling output constraints is more complicated; 
we refer to \cite{Haberleverena2021}, \cite{Bernstein2019}, \cite{colot2024optimalpowerflowpursuit} for some possible approaches.

\section{Numerical examples} \label{sec:numerics}

We present two simple numerical examples to illustrate our theoretical results. The code to generate the examples is  available at \url{https://github.com/mattiabianchi/OFO_ECC25.git}

\subsection{Linear plant without input constraints}\label{sec:linearsystem}

We consider the linear time-invariant plant $\dot x = Ax+Bu+B_w w$, $y=Cx$, where $w\in \R$ is an unmeasured disturbance, 
\begin{align*}
	A =\begin{bmatrix}
		-1 & 10 
		\\
		- 10 & -1
	\end{bmatrix},  B = \begin{bmatrix}
		0 \\ 1
	\end{bmatrix},  B_w = \begin{bmatrix} 1 \\ 1 \end{bmatrix},   C = \begin{bmatrix}
		1 & 0
	\end{bmatrix}.
\end{align*} The cost function is given by $\Phi(y,u) = 0.01 u^2+y^2$.

First, assume that $w$ is constant. The matrix $A$ is Hurwitz, and  the plant has steady-state map $s(u) = - A^{-1}(Bu+B_w w)$.  
\Cref{asm:simplifying} holds with $W(x,u) = \|x-s(u)\|^2$, $c_3 =d_3 = \mu_3 = \zeta_3= 1$ (this can be checked by solving the Lyapunov equation for $A$), $\mu_\Phi = 0.02$,  $\ell_{\Phi_u} = 0$, $\ell_{\Phi_y} = 2 \ell_h $ (where $\ell_h = \|CA^{-1}B\|$ is the Lipschitz constant of $h(u)$), $\ell_g = 1$. The bound in \eqref{eq:muBound} gives $\mu_\Phi > 0.0198$ and it is satisfied. Hence, by \Cref{cor:first}, the \gls{OFO} controller   \eqref{eq:OFO} guarantees global exponential stability  and convergence to the  solution of problem \eqref{eq:opt1}, for any value of $\alpha$; for comparison, the result in \cite[Cor.~3.3]{Hauswirthadrian2021} would only ensure stability under the upper bound $\alpha < 10.1$. Note that the presence of the \emph{constant} disturbance  $\omega$ (i.e., the constant offset in $f(x,u)$) does not affect the result in \Cref{cor:first}. Furthermore,  implementing \eqref{eq:OFO} does not require measuring $w$, since $\nabla h(u) = -CA^{-1}B $ is independent of $\omega$. 

Next, to evaluate how the controller can respond to time-varying conditions, we take the disturbance as piecewise constant and switching between the values $w = -10$ and $w = 10$.  The resulting input trajectories, for different values of $\alpha$, are shown in \Cref{fig:1}. Note that the  optimal input $u^\star$ that solves \eqref{eq:opt2} depends on $\omega$, so it is  piecewise constant (as shown by the blue line in \Cref{fig:1}). \Cref{cor:first} ensures exponential stability to the optimal solution in every interval where $w$ is constant.  For small values of $\alpha$, the controller update is  slow and it does not track accurately the optimal input $u^\star$. For $\alpha = 100$, the input quickly converges to the optimum after the disturbance changes. Finally, for larger values of $\alpha$, convergence is retained but oscillatory behavior and overshooting are observed; thus excessively large values for the control gain might be undesirable.

\begin{figure}
	\includegraphics[width=0.95\columnwidth]{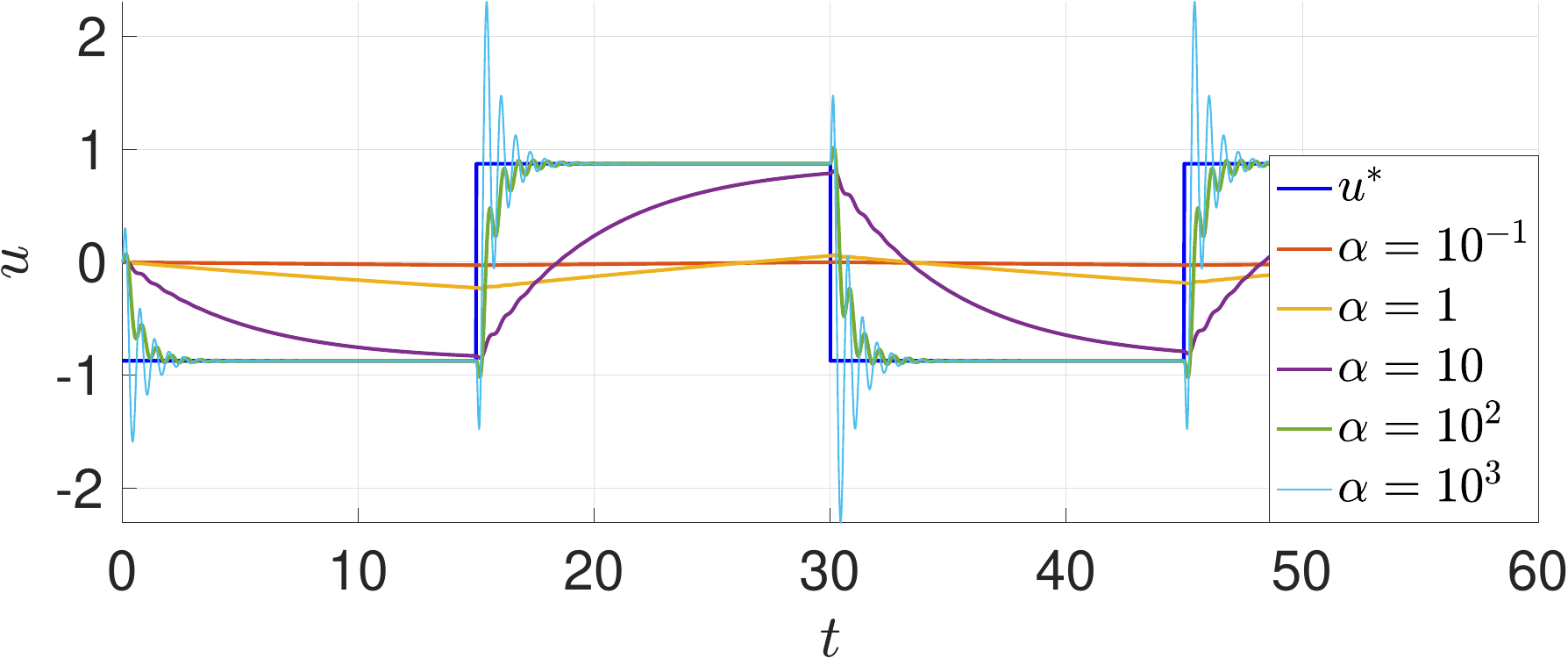}
	\caption{Performance of the \gls{OFO} controller in \eqref{eq:OFO} for the linear plant in \Cref{sec:linearsystem} and different values of $\alpha$.} \label{fig:1}
\end{figure}
\begin{figure}
	\includegraphics[width=0.95\columnwidth]{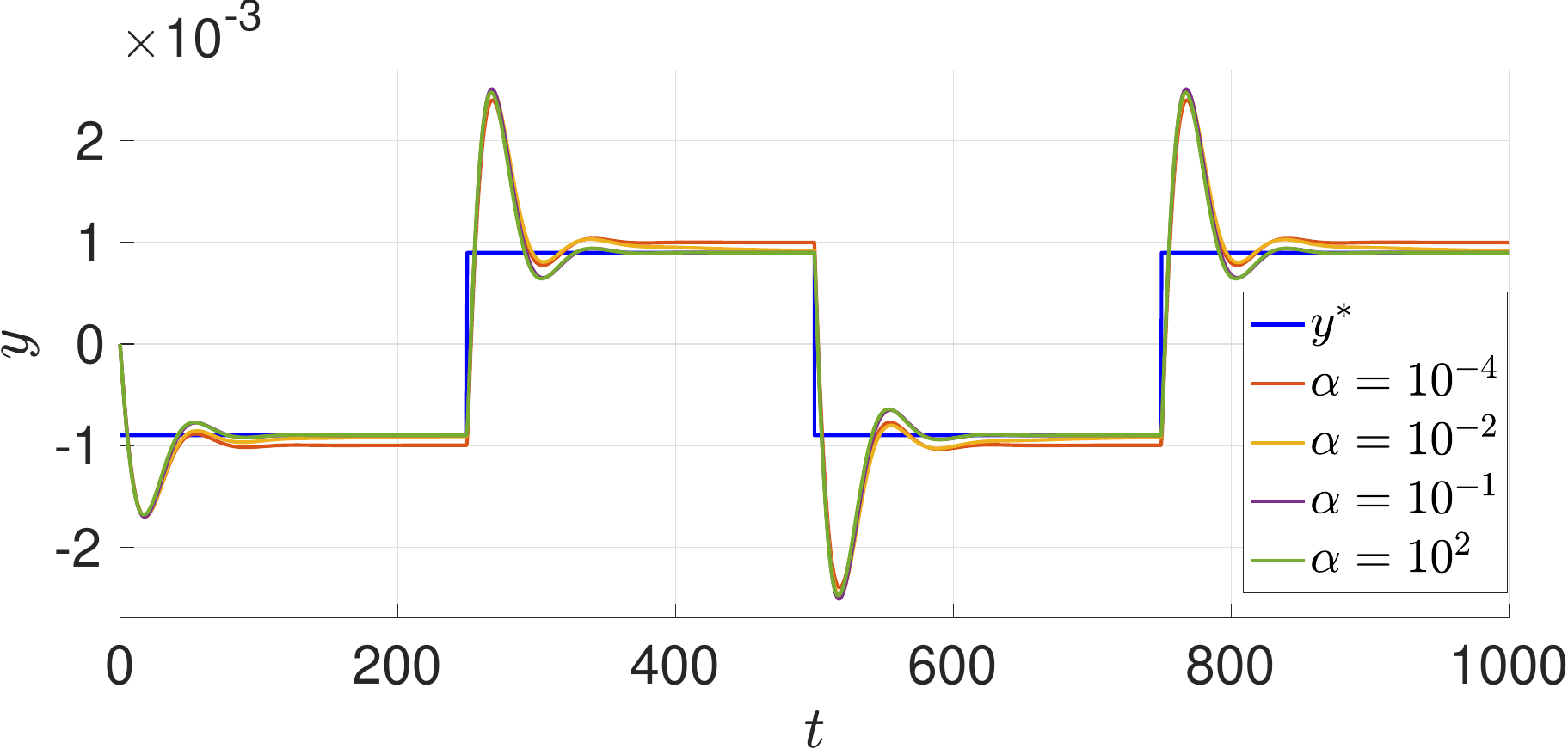}
	\\
	\includegraphics[width=0.95\columnwidth]{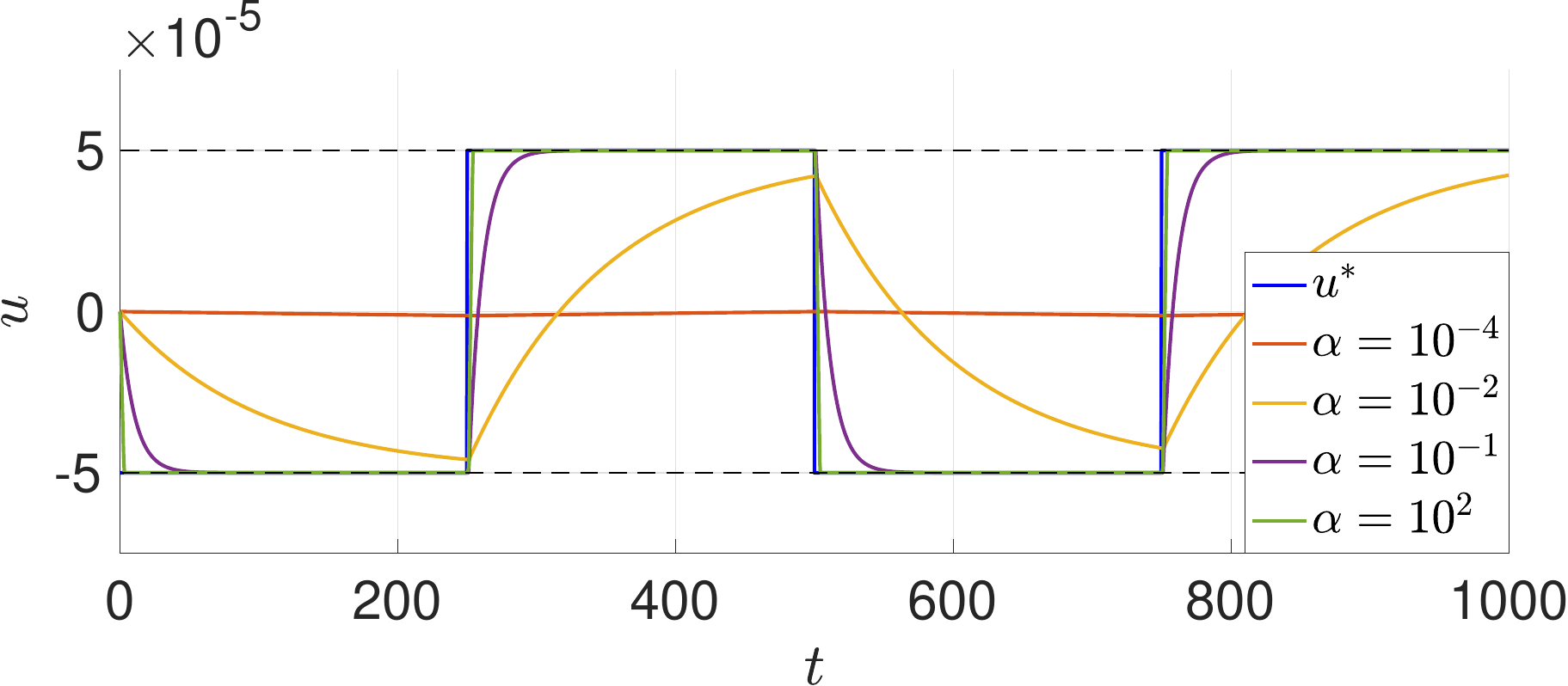}
	\caption{The plant in \Cref{sec:nonlinear} in closed-loop with \eqref{eq:OFOConstrainedSmooth}. Input constraints (dotted lines) are always satisfied.} \label{fig:2}
\end{figure}

\subsection{Nonlinear plant with input constrain}\label{sec:nonlinear}

We consider the plant $\dot x = Ax+ B(u+\operatorname{sin}(u))+B_w w$, $y = Cx$, where $w$ is an unmeasured disturbance,
\begin{align*}
	A = \begin{bmatrix}
		0 & -0.1 \\ 0.1 & -0.1
	\end{bmatrix}, B=\begin{bmatrix}
		0 \\ 0.1
	\end{bmatrix}, B_w = \begin{bmatrix}
		0.1 \\0.1
	\end{bmatrix}, C=\begin{bmatrix}
		1 & 1
	\end{bmatrix}.
\end{align*}
The cost function is given by $\Phi(u,y) = 11 u^2+\sqrt{y^2+1}$. The input is constrained  as $-5\cdot 10^{-5} \leq u \leq 5 \cdot {10^{-5}}$. 

First, assume that $\omega$ is constant. \Cref{asm:core}(a) holds with $\ell_f = 2\|B\|$ (by computing $\nabla_u f(x,u)$, since $\cos(u) \leq 1$), $\ell_g  =\|C\|$.
The matrix $A$ is Hurwitz, hence, for any constant disturbance $w$, the plant has a steady-state map $s(u) = -A^{-1}(Bu+B\operatorname{sin}(u)+B_w w)$. We solve the Lyapunov inequality $A P +PA^\top + 0.045 P \prec 0$. Given the solution  $P = \begin{smallbmatrix}
	0.66 &  0.33 \\ 0.33 & 0.66
\end{smallbmatrix}$, we define the Lyapunov function $W(x,u) = (x-s(u))^\top P (x-s(u))$ in \Cref{asm:simplifying}. Note that $\nabla_x W(x,u)^\top f(x,u) = 2(x-s(u))^\top P(Ax+Bu+B\operatorname{sin}(u) +B_w w) = (x-s(u))^\top (PA + PA^\top)(x-s(u)) \leq - 0.045 (x-s(u))^\top P(x-s(u))$. Thus, \Cref{asm:simplifying}(b) holds with $\mu_3 = 0.45 \lambda_{\min}(P)$, $d_3=\zeta_3 = \lambda_{\max}(P)$, $c_3 = \lambda_{\min}(P)$. Also, $\| \nabla h(u) \| = \| CA^{-1}B(1+\operatorname{cos} u) \| \leq 2\| CA^{-1}B\| \coloneqq \ell_{h}$ is bounded and $\sqrt{y^2+1}$ is $1$-Lipschitz continuous, so \Cref{asm:simplifying}(c) holds with $\ell_{\Phi_y} = \ell_h$. Similarly, $\| \nabla\sqrt{y^2+1} \| = \|y/\sqrt{y^2+1}\| \leq 1$ and $\nabla h(u)$ is $\ell_{\nabla h}$-Lipschitz continuous, with $\ell_{\nabla h} = \|CA^{-1}B\|$; therefore \Cref{asm:simplifying}(d) holds with $\ell_{\Phi_u} = \ell_{\nabla h}$. Finally, $\mu_\Phi = 22$, the bound in \eqref{eq:muBound} is satisfied, and hence the controller \eqref{eq:OFOConstrainedSmooth} guarantees exponential stability of the plant for any value of $\alpha$ (and any \emph{constant} disturbance $\omega$) by \Cref{cor:projected}. 

We test the controller  \eqref{eq:OFOConstrainedSmooth} for different values of $\alpha$, with periodic piecewise constant disturbance, switching between $w = -0.001$ and $w = 0.001$. For smaller values of $\alpha$, the controller is slower in adapting to changes in the disturbance, and the output slower in reaching the optimal steady state. Interestingly, for  larger values of $\alpha$, the \gls{OFO} scheme in \eqref{eq:OFOConstrainedSmooth} approximately behaves as a bang-bang controller, where the input quickly switches  between its lower and upper bounds. 

\section{Conclusion and research directions}\label{sec:conclusion}

We showed that, under suitable assumptions, online feedback optimization guarantees stability without requiring any timescale separation. 
Although we framed our results in continuous-time, discrete-time \gls{OFO} controllers can be studied with analogous assumptions and convergence results. Our analysis can be also extended to tracking  the equilibrium trajectory for  problems with time-varying disturbances and costs, where our results can be beneficial in achieving better tracking guarantees, by allowing for larger control gains.    

Future research should focus on implementing more realistic simulation scenarios, to validate the applicability of our results. Further, in this paper we only focused on the analysis of an existing \gls{OFO} scheme. Developing new control strategies for \gls{OFO}, specifically designed to  avoid the need for timescale separation, and under less restrictive assumptions, is a significant direction for future work.

%

\bibliographystyle{elsarticle-num} 
\bibliography{library}

\end{document}